\documentclass[11pt]{amsart}
\usepackage[margin=1.2 in]{geometry}
\usepackage{amssymb}
\usepackage{setspace}
\usepackage{indentfirst}
\usepackage{multicol}
\usepackage{amssymb,amsmath,amsthm,mathtools}
\usepackage{physics}
\usepackage{float}
\usepackage[english]{babel}
\usepackage{tikz-cd}

\newtheorem{thm}{Theorem}[section]
\newtheorem{prop}[thm]{Proposition}
\newtheorem{lem}[thm]{Lemma}
\newtheorem{cor}[thm]{Corollary}
 
\theoremstyle{definition}
\newtheorem{example}[thm]{Example}

\newtheorem{remark}[thm]{Remark}

\numberwithin{equation}{section}

\newcommand{\wt}{\widetilde}

\newcommand{\GL}{\operatorname{GL}}
\newcommand{\PGL}{\operatorname{PGL}}
\newcommand{\Aut}{\operatorname{Aut}}
\newcommand{\Spec}{\operatorname{Spec}}
\newcommand{\id}{\operatorname{id}}
\newcommand{\Pic}{\operatorname{Pic}}

\newcommand{\im}{\operatorname{im}}

\newcommand{\Hom}{\operatorname{Hom}}
\newcommand{\Gal}{\operatorname{Gal}}
\newcommand{\Dic}{\operatorname{Dic}}

\newcommand{\beq}{\begin{equation}}
\newcommand{\eeq}{\end{equation}}
\usepackage[
            pdftex,
           colorlinks=true,
            urlcolor=blue,       % \href{...}{...} external (URL)
            filecolor=blue,     % \href{...} local file
            linkcolor=blue,       % \ref{...} and \pageref{...}
            citecolor=blue,         % \cite{...}
            pdftitle= {Title},
            pdfauthor={Author},
            pdfsubject={Subject},
            pdfkeywords={Key}
            pagebackref,
            bookmarksopen=true]
            {hyperref}
\usepackage{hyperref}

\begin{document}

\title[Groups acting on cubic surfaces in characteristic zero]{Groups acting on Cubic Surfaces \\ in Characteristic Zero}
\author{Jonathan M. Smith}

\abstract{}
For every field $k$ of characteristic zero, we determine the groups that act as automorphisms on a smooth cubic surface over $k$. We also determine the groups that act on $k$-rational, stably $k$-rational, or $k$-unirational smooth cubic surfaces.
\endabstract{}

\maketitle

\section{Introduction}

The purpose of this paper is to determine, for each field $k$ of characteristic zero, the groups that act by automorphisms on smooth cubic surfaces over $k$. We also determine the groups that act on a $k$-rational, stably $k$-rational, or $k$-unirational cubic surface. This is progress toward the goal of classifying the finite subgroups of the plane Cremona group over an arbitrary field of characteristic zero.

The first attempts at a classification of the automorphism groups of cubic surfaces over an algebraically closed field of characteristic zero are due to S. Kantor \cite{Kan95}, A. Wiman \cite{Wim96}, and B. Segre \cite{Seg42}, but the first complete classification is due to T. Hosoh \cite{Hos97}. More generally, the automorphism groups of del Pezzo surfaces of any degree over an algebraically closed field of characteristic zero were computed by I. Dolgachev and V. Iskovskikh in \cite{DolIsk09}. A classification of the automorphism groups of del Pezzo surfaces over an algebraically closed field of any characteristic was obtained more recently by I. Dolgachev, A. Duncan, and G. Martin in \cite{DolDun18, DolMar22, DolMar23}. 

However, the classification over non-algebraically closed fields is still largely open. The automorphisms of real del Pezzo surfaces were studied by E. Yasinsky in \cite{Yas19}. The automorphisms of quintic del Pezzo surfaces over any perfect field were classified independently by A. Boitrel \cite{Boi23} and A. Zaitsev \cite{Zai23}. The largest automorphism groups of smooth cubic surfaces over finite fields of characteristic 2 were determined by A. Vikulova in \cite{Vik23}. Lastly, the author determined the maximal automorphism groups of quartic del Pezzo surfaces over any field of characteristic zero in \cite{Smi23}. This paper is fundamentally an extension of the results in \cite{Smi23} to smooth cubic surfaces. 

For a smooth cubic surface $X$ over any field $k$, the action of $\Aut(X)$ on the 27 lines of $X_{\bar{k}}$ yields an injective group homomorphism $\Aut(X) \xhookrightarrow{} W(\mathsf{E}_6)$ into the Weyl group of the root system $\mathsf{E}_6$ (see Chapter 25 of \cite{Man86}). This map identifies the automorphism group of a cubic surface with the conjugacy class of a subgroup of $W(\mathsf{E}_6)$. This identification has at least two benefits: (1) it enables us to compare automorphisms of cubic surfaces across various fields, and (2) the action of $\Aut(X)$ on the lines of $X_{\bar{k}}$ is useful when analyzing the rationality of $X$. 

We say a subgroup $G$ of $W(\mathsf{E}_6)$ \textit{acts by automorphisms} on a smooth cubic surface $X$ if $G$ is contained in the image of the map $\Aut(X) \xhookrightarrow{} W(\mathsf{E}_6)$. This notion is well-defined up to conjugacy in $W(\mathsf{E}_6)$. For any field $k$, let $\mathcal{P}_{3,k}$ be the collection of conjugacy classes of subgroups of $W(\mathsf{E}_6)$, partially ordered by inclusion, that act by automorphisms on some smooth cubic surface over $k$. We let $\epsilon_n$ denote a primitive $n$th root of unity. Our first main result completely describes $\mathcal{P}_{3,k}$ for any field $k$ of characteristic zero.

\begin{thm}
\label{thm:main}
Let $k$ be a field of characteristic zero. If $G$ is a maximal group in $\mathcal{P}_{3,k}$, then $G$ is one of the groups in the table. Each group appears in $\mathcal{P}_{3,k}$ if and only if the condition in the third column is satisfied. When the condition in the third column is satisfied, the fourth column provides a smooth cubic surface on which the group acts.
\begin{center}
    \begin{tabular}{|c|c|c|c|}
        \hline
        \textbf{Name} & \textbf{Structure} & \textbf{Condition on } $k$ & \textbf{Surface} \\
        \hline
        $5\mathsf{A}$ & $S_5$ & none & Eq. \ref{eq:clebsch} \\
        \hline
        $3\mathsf{C}$ & $C_3^3 \rtimes S_4$ & $\epsilon_3 \in k$ & Eq. \ref{eq:fermat} \\
        \hline
        $12\mathsf{A}$ & $\mathcal{H}_3(3) \rtimes C_4$ & $\epsilon_{12} \in k$ & Eq. \ref{eq:H33C4} \\
        \hline
        $8\mathsf{A}$ & $C_8$ & $\epsilon_8 \in k$ & Eq. \ref{eq:8A} \\
        \hline
        $4\mathsf{A}$ & $C_4$ & $i \in k$ & Eq. \ref{eq:8A} \\
        \hline
        $3\mathsf{C}_1$ & $C_3^2 \rtimes D_8$ & none & Eq. \ref{eq:C32D8} \\
        \hline
        $3\mathsf{C}_2$ & $\Dic_{12}$ & $x^2 + y^2 = -3$ has a solution over $k$ & Eq. \ref{eq:C3C4} \\
        \hline
    \end{tabular}
    \end{center}
Moreover, each group in the table is maximal in $\mathcal{P}_{3,k}$ for some choice of $k$.
\end{thm}

\begin{remark}
    By construction, if $G$ is in $\mathcal{P}_{3,k}$ and $H$ is a subgroup of $G$, then $H$ is in $\mathcal{P}_{3,k}$. Therefore, $\mathcal{P}_{3,k}$ is the downward closure of the groups in the table that are realized over $k$. The fact that each group in the table is maximal for some $k$ indicates that the list of groups is as small as possible. The groups $5\mathsf{A}$, $3\mathsf{C}$, $12\mathsf{A}$, and $8\mathsf{A}$ from the table are the well known maximal automorphism groups of cubic surfaces over $k = \bar{k}$ (cf. Table 9.6 of \cite{Dol12}). The group $4\mathsf{A}$ is not maximal when $k = \bar{k}$, but it is maximal over, for example, $\mathbb{Q}(i)$. The groups $3\mathsf{C}_1$ and $3\mathsf{C}_2$ are subgroups of $3\mathsf{C}$ that are realized on $k$-forms of the Fermat cubic surface.
\end{remark}

Our second main result determines when a subgroup of $W(\mathsf{E}_6)$ acts on a smooth cubic surface over $k$ and yet does \textit{not} act on a $k$-rational or stably $k$-rational smooth cubic surface. There is only one group that exhibits this phenomenon.

\begin{thm}
\label{thm:rationality}
Let $k$ be a field of characteristic zero. Let $G$ be a subgroup of $W(\mathsf{E}_6)$ that acts by automorphisms on a smooth cubic surface over $k$. Then $G$ acts by automorphisms on a $k$-rational smooth cubic surface unless all three of the following conditions are satisfied:
\begin{itemize}
    \item[(i)] $G$ is conjugate to $3\mathsf{C}_2$ in $W(\mathsf{E}_6)$.
    \item[(ii)] $k$ does not contain $\epsilon_3$.
    \item[(iii)] $x^2 + y^2 = -3$ has a solution over $k$.
\end{itemize}
If these three conditions are satisfied, then $G$ does not act on a $k$-rational or stably $k$-rational smooth cubic surface.
\end{thm}

\begin{remark}
There are exactly two groups, under the same field conditions, that act on a quartic del Pezzo surface and yet do not act on a $k$-rational or stably $k$-rational surface \cite{Smi23}. Nonetheless, for any field $k$ of characteristic zero, any group that acts by automorphisms on a smooth cubic surface over $k$ must act by automorphisms on a $k$-unirational cubic surface (see Corollary \ref{cor:unirational}). This extends a similar result obtained in \cite{Smi23} for quartic del Pezzo surfaces.
\end{remark} 

The paper is structured as follows. Section 2 provides useful background information on del Pezzo surfaces and the already known classification of automorphisms of cubic surfaces over algebraically closed fields. Section 3 is the heart of the paper and devoted to proving Theorem \ref{thm:main}. Section 4 addresses the rationality of surfaces exhibiting various group actions, culminating in the proof of Theorem \ref{thm:rationality}.

\subsection*{Acknowledgments} The author would like to thank Alexander Duncan for suggesting this problem and offering helpful comments. This work was partially supported by a SPARC Graduate Research Grant from the Office of the Vice President for Research at the University of South Carolina.

\section{Preliminaries}

Unless stated otherwise, $k$ will be a field of characteristic zero. If $X$ is a variety over $k$, we let $\overline{X}$ denote $X \times \Spec \bar{k}$. Then $\overline{X}$ has an action of $\Gal(\bar{k}/k)$ induced by the action on the second factor. If $X$ and $Y$ are varieties over $k$, we say that a rational map $X \dashrightarrow Y$ is defined over $k$ if the corresponding map $\overline{X} \dashrightarrow \overline{Y}$ is $\Gal(\bar{k}/k)$-equivariant. We let $\Aut(X)$ denote the automorphisms of $X$ defined over $k$, while $\Aut(\overline{X})$ denotes the automorphisms defined over $\bar{k}$. A surface $X$ is \textit{$k$-rational} if there exists a birational map defined over $k$ from $X$ to $\mathbb{P}^2_k$. A surface $X$ is \textit{stably $k$-rational} if $\mathbb{P}^n_k \times X$ is $k$-rational for some $n \geq 0$. A surface $X$ is \textit{$k$-unirational} if there is a dominant rational map $\mathbb{P}^n_k \dashrightarrow X$ defined over $k$. We let $X(k)$ denote the set of $k$-rational points on $X$. 

\subsection{Group theoretic notation} Throughout, we adopt the following conventions:
\begin{itemize}
    \item $C_n$ and $D_n$ denote the cyclic group and dihedral group of order $n$ respectively.
    \item $S_n$ is the symmetric group acting on $n$ elements.
    \item $C_n^m$ is the direct sum of $m$ copies of $C_n$.
    \item $\Dic_{n}$ denotes the dicyclic group of order $n$.
    \item $\mathcal{H}_n(p)$ is the Heisenberg group of upper triangular $n \times n$ matrices with entries in $\mathbb{F}_p$ where each diagonal entry is 1.
    \item $A \rtimes B$ denotes a semidirect product of $A$ and $B$.
    \item $W(R)$ denotes the Weyl group of a root system $R$.
    \item $\epsilon_n$ denotes a primitive $n$th root of unity.
\end{itemize}

\subsection{Del Pezzo surfaces} 

Recall that a \textit{del Pezzo surface} $X$ is a smooth projective surface on which the anticanonical bundle $\omega_{X}^{-1}$ is ample. The \textit{degree} of a del Pezzo surface is defined to be $d = (\omega_{X}^{-1}, \omega_{X}^{-1})$ where $(\, , \, )$ denotes the intersection pairing on $\Pic X$. If $d \geq 3$, then the sheaf $\omega_{X}^{-1}$ is very ample, and the sections of $\omega_{X}^{-1}$ embed $X$ into a projective space of dimension $d$. Under the anticanonical embedding, del Pezzo surfaces of degree 3 are identified with smooth cubic surfaces in $\mathbb{P}^3$. The reader can consult \cite{Dem80}, \cite{Man86}, or \cite{Dol12} for a more complete exposition of del Pezzo surfaces.

Let $k$ be a field of characteristic zero, and let $X$ be a del Pezzo surface of degree $d$ over $k$. If $\overline{X}$ is not isomorphic to $\mathbb{P}^1 \times \mathbb{P}^1$, then $\overline{X}$ can be obtained as a blow up $\pi: \overline{X} \to \mathbb{P}^2$ of $r = 9-d$ points $\{P_1,...,P_r\}$ in general position. Consequently, $\Pic \overline{X} \cong \mathbb{Z}^{r+1}$, and $\Pic \overline{X}$ is generated by $\{\pi^{-1} H, E_1, ..., E_{r}\}$ where $H$ is a line on $\mathbb{P}^2$ and $E_i$ is the exceptional divisor corresponding to $P_i$. Then $K_X = -3H + E_1 + ... + E_{r}$ is a canonical divisor on $\overline{X}$, and the intersection pairing on $\Pic \overline{X}$ is determined by the rules
\begin{equation*}
    (H,H) = 1, \quad (H,E_i) = 0 \text{ and } (E_i,E_i) = -1 \text{ for $1 \leq i \leq r$}, \quad (E_i,E_j) = 0 \text{ for $i \neq j$}.
\end{equation*}
Following \cite{Man86}, we define
\begin{equation*}
    R_r = \{l \in \Pic \overline{X} \mid (l,K_X) = 0, \, (l,l) = -2\} \text{ and } I_r = \{l \in \Pic \overline{X} \mid (l,K_X) = (l,l) = -1\}.
\end{equation*}
An irreducible curve $D$ on $\overline{X}$ is \textit{exceptional} if $D \cong \mathbb{P}^1$ and $(D,D) = -1$. The map $D \mapsto \mathcal{O}_X(D)$ is a bijection from the set of exceptional curves on $\overline{X}$ to $I_r$. When $\overline{X}$ is a cubic surface, the exceptional curves on $\overline{X}$ are precisely the 27 lines on $\overline{X}$ with respect to the anticanonical embedding.

The orthogonal complement of $K_X$ in $\mathbb{R} \otimes \Pic \overline{X}$ can be identified with a Euclidean vector space of dimension $r$, and $R_r$ forms a root system of rank $r$. The actions of both $\Gal(\bar{k}/k)$ and $\Aut(\overline{X})$ preserve the intersection pairing, and therefore permute the elements of $R_r$ and $I_r$. Up to a choice of basis, the actions induce group homomorphisms $\rho: \Gal(\bar{k}/k) \to W(R_r)$ and $\tau: \Aut(\overline{X}) \to W(R_r)$, where $W(R_r)$ denotes the Weyl group of the root system $R_r$. If we obtain maps $\rho'$ and $\tau'$ by selecting a different basis, then the images of $\rho$ and $\rho'$ (resp. $\tau$ and $\tau'$) are conjugate in $W(R_r)$. When $d \leq 5$, the map $\tau: \Aut(\overline{X}) \to W(R_r)$ is injective, so we can identify automorphism groups of del Pezzo surfaces of degree $d \leq 5$ with conjugacy classes of subgroups in $W(R_r)$. 

For cubic surfaces, $R_r$ is the root system $\mathsf{E}_6$, so automorphism groups of cubic surfaces correspond to subgroups of $W(\mathsf{E}_6)$. As stated previously, a subgroup $G$ in $W(\mathsf{E}_6)$ \textit{acts by automorphisms} on a cubic surface $X$ over $k$ if the image of $\Aut(X)$ in $W(\mathsf{E}_6)$ contains $G$ up to conjugacy. We say that $G$ is \textit{realized over $k$} if $G$ acts by automorphisms on some smooth cubic surface over $k$.

\subsection{Automorphisms over algebraically closed fields} When $k = \bar{k}$, the subgroups of $W(\mathsf{E}_6)$ that act by automorphisms on a smooth cubic surface over $k$ have been classified. See Chapter 9 of \cite{Dol12} or Section 6.5 of \cite{DolIsk09} for fields of characteristic zero, and see \cite{DolDun18} for fields of positive characteristic. The reader should refer to Table 7 and Table 8 in \cite{DolDun18} for the structure of these groups and for normal forms of the cubic surfaces on which they act.

The stratification of cubic surfaces with various geometric group actions over an algebraically closed field $k$ of characteristic zero is useful for organizing our classification. Following \cite{DolDun18}, we let $\mathcal{M}_{\text{cub}}(k)$ be the coarse moduli space of smooth cubic surfaces over $k$. The conjugacy classes of elements in $W(\mathsf{E}_6)$ are labeled as $1A, ..., 12A$ where the number in the label is the order of the element. To each conjugacy class in $W(\mathsf{E}_6)$, there is a corresponding subvariety of $\mathcal{M}_{\text{cub}}(k)$ consisting of the cubic surfaces on which that class acts by automorphisms. This process produces a stratification of $\mathcal{M}_{\text{cub}}(k)$ by conjugacy classes of $W(\mathsf{E}_6)$, which we have taken from \cite{DolDun18} for convenience.

\vspace{-5mm}
\begin{figure}[h!]
\label{figure:1}
\caption{Specialization of strata in $\mathcal{M}_{\text{cub}}$ for $k = \bar{k}$ with $\text{char}(k) = 0$.}
\begin{tikzcd}
& 1A \arrow[d] &  &  \\
& 2A \arrow[ld] \arrow[d] \arrow[rrdd] &  & \\
2B \arrow[d] \arrow[rd] & 3D \arrow[d] \arrow[ld] \arrow[rd] & &  \\
4B \arrow[d] \arrow[rd] & 6E \arrow[d] \arrow[ld] & 3A \arrow[ld] \arrow[d] & 4A \arrow[ld] \arrow[d] \\
5A & 3C & 12A & 8A                     
\end{tikzcd}
\end{figure}
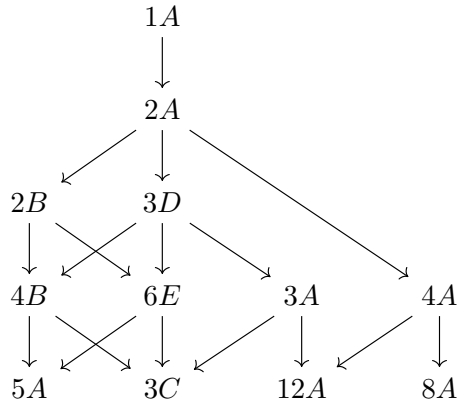

Surprisingly, the stratification of $\mathcal{M}_{\text{cub}}(k)$ by conjugacy classes of subgroups of $W(\mathsf{E}_6)$ is identical to the stratification by conjugacy classes of elements. For this reason, the symbols of Figure 1 are also used in \cite{DolDun18} to label the full automorphism groups of smooth cubic surfaces up to conjugacy in $W(\mathsf{E}_6)$. We will differentiate between the conjugacy class of the element and the conjugacy class of the associated automorphism group with a change of font. For example, $3C$ will denote the conjugacy class of the element, while $3\mathsf{C}$ will denote the conjugacy class of the corresponding automorphism group.

\section{The Classification} 

In this section, for every field $k$ of characteristic zero, we determine the subgroups of $W(\mathsf{E}_6)$ that act by automorphisms on a smooth cubic surface over $k$. From Figure 1, these subgroups must be contained in $5\mathsf{A}$, $3\mathsf{C}$, $12\mathsf{A}$, or $8\mathsf{A}$. A list of elements in $W(\mathsf{E}_6)$, up to conjugacy, that act on a smooth cubic surface $X$ over $\bar{k}$ is contained in Table 9.5 of \cite{Dol12}. We can identify $\Aut(\overline{X})$ with a subgroup of $\PGL_4(\bar{k})$ under the anticanonical embedding. Let $g$ be an element of $W(\mathsf{E}_6)$ up to conjugacy, and suppose $g$ acts by automorphisms on a smooth cubic surface over $\bar{k}$. The associated element of $\PGL_4(\bar{k})$ by which $g$ may act is determined up to a choice of coordinates in Section 9.5 of \cite{Dol12}. If $g$ acts on a smooth cubic surface over $k$, then the associated element of $\PGL_4(\bar{k})$ must be similar to an element of $\PGL_4(k)$.

\begin{lem}
\label{lem:elements}
Let $k$ be a field of characteristic zero, and let $g$ be an element of $W(\mathsf{E}_6)$ that acts by automorphisms on a smooth cubic surface over $k$.
\begin{itemize}
    \item[(i)] If $g$ is an element of type $3A$, $6A$, $6C$, or $9A$, then $\epsilon_3 \in k$.

    \item[(ii)] If $g$ is an element of type $4A$, then $i \in k$.

    \item[(iii)] If $g$ is an element of type $8A$, then $\epsilon_8 \in k$.

    \item[(iv)] If $g$ is an element of type $12A$, then $\epsilon_{12} \in k$.
\end{itemize}
\end{lem}
\begin{proof}
First note that an element $P \in \PGL_4(\bar{k})$ is similar to an element $Q \in \PGL_4(k)$ if and only if the rational canonical form of a lift of $P$ to $\GL_4(\bar{k})$ is defined over $k$.
    
(i) By Theorem 10.4 of \cite{DolDun18}, an element of type $3A$ acts by the diagonal matrix $M = [\epsilon_3,1,1,1]$ up to projective equivalence. Let $\wt{M} = [a\epsilon_3,a,a,a]$ be a lift of $M$ to $\GL_4(\bar{k})$. Similarly, an element of type $6C$ acts by the diagonal matrix $N = [1,1,\epsilon_6^4, \epsilon_6]$ up to projective equivalence by Section 9.5.1 of \cite{Dol12}. Let $\wt{N} = [a,a,a\epsilon_6^4,a\epsilon_6]$ be a lift to $\GL_4(\bar{k})$. We calculate the rational canonical forms:
\begin{equation*}
    \wt{M}_{RCF} = \begin{bmatrix}
    a & 0 & 0 & 0 \\
    0 & a & 0 & 0 \\
    0 & 0 & 0 & -a^2\epsilon_3 \\
    0 & 0 & 1 & a(\epsilon_3 + 1)
    \end{bmatrix}
    \quad
    \text{and}
    \quad
    \wt{N}_{RCF} = \begin{bmatrix}
        a & 0 & 0 & 0 \\
        0 & 0 & 0 & -a^3\epsilon_6^2 \\
        0 & 1 & 0 & a^2\epsilon_6^2 \\
        0 & 0 & 1 & a
    \end{bmatrix}.
\end{equation*}
If the rational canonical form in either case is defined over $k$, then $\epsilon_3 \in k$. By Table 4 of \cite{DolDun18}, elements of type $6A$ or $9A$ have powers of type $3A$, so if an element of type $6A$ or $9A$ is realized over $k$, then $\epsilon_3 \in k$.

(ii) By Lemma 11.1 of \cite{DolDun18}, an element of type $4A$ acts by the diagonal matrix $[i,-1,1,1]$ up to projective equivalence. We calculate the rational canonical form of an arbitrary lift and conclude that $i \in k$.

(iii) The cyclic group generated by an element of type $8A$ in $W(\mathsf{E}_6)$ acts by the group generated by the diagonal matrix $[1,\epsilon_8^6, \epsilon_8, \epsilon_8^4]$ up to projective equivalence by Lemma 12.12 of \cite{DolDun18}. We calculate the rational canonical form of an arbitrary lift and conclude that $\epsilon_8 \in k$.

(iv) By Table 4 of \cite{DolDun18}, an element of type $12A$ has powers of type $3A$ and $4A$, so that $\epsilon_3 \in k$ and $i \in k$. 
\end{proof}

\begin{remark}
    Lemma \ref{lem:elements} is used to prove Theorem \ref{thm:main}, but we can make a stronger claim regarding the elements of $W(\mathsf{E}_6)$ after proving Theorem \ref{thm:main}; see Corollary \ref{cor:elements}.
\end{remark}

\subsection{Subgroups of \texorpdfstring{$5\mathsf{A}$}{5A}, \texorpdfstring{$8\mathsf{A}$}{8A}, and \texorpdfstring{$12\mathsf{A}$}{12A}} Recall that the \textit{Clebsch cubic surface} is the smooth cubic surface given by the equation
\begin{equation}
\label{eq:clebsch}
    \left(\sum_{i \neq j} x_i^2x_j \right) + 2 \left(\sum_{i < j < k} x_ix_jx_k \right) = 0
\end{equation}
in $\mathbb{P}^3$. Over $\bar{k}$, the automorphism group of the Clebsch cubic surface is isomorphic to $S_5$ and its image in $W(\mathsf{E}_6)$ is the group $5\mathsf{A}$. 

\begin{prop}
\label{prop:5A}
    Let $G$ be a subgroup of $5\mathsf{A}$ up to conjugacy in $W(\mathsf{E}_6)$. Then $G$ acts by automorphisms on the Clebsch cubic surface over any field $k$ of characteristic zero.
\end{prop}
\begin{proof}
    Let $X$ be the Clebsch cubic surface. Then $\Aut(X) = \Aut(\overline{X})$ since $\Aut(\overline{X})$ is generated by the morphisms
    \begin{align*}
       &g_1: (x_0:x_1:x_2:x_3) \mapsto (-x_0 - x_1 - x_2 - x_3: x_0 : x_1: x_2) \\
       &g_2: (x_0 : x_1:x_2:x_3) \mapsto (x_1: x_0:x_2:x_3)
    \end{align*}
    defined over $k$.
\end{proof} 

Over $\bar{k}$, the group $8\mathsf{A}$ acts by automorphisms on the surface given by
\begin{equation}
    \label{eq:8A}
    x_0^3 + x_0x_3^2 - x_1x_2^2 + x_1^2x_3 = 0
\end{equation}
and this surface is unique up to projective equivalence (see Lemma 12.12 of \cite{DolDun18}).

\begin{prop}
\label{prop:8A}
Let $k$ be a field of characteristic zero. Let $G$ be a subgroup of $8\mathsf{A}$ up to conjugacy in $W(\mathsf{E}_6)$.
\begin{itemize}
    \item[(i)] If $G$ is trivial or $2\mathsf{A}$, then $G$ is realized over $k$.
    \item[(ii)] If $G$ is $4\mathsf{A}$, then $G$ is realized over $k$ if and only if $i \in k$.
    \item[(iii)] If $G$ is $8\mathsf{A}$, then $G$ is realized over $k$ if and only if $\epsilon_8 \in k$.
\end{itemize}
In each case, $G$ acts by automorphisms on the cubic surface given by Equation \ref{eq:8A}.
\end{prop}
\begin{proof}
    The forward directions follow immediately from Lemma \ref{lem:elements}. For the reverse directions, let $X$ be the surface given by Equation \ref{eq:8A}. Then $\Aut(\overline{X})$ is generated by the diagonal morphism $g = [1, \epsilon_8^6, \epsilon_8, \epsilon_8^4]$. Over any $k$, the group $\langle g^4 \rangle$ realizes $2\mathsf{A}$. If $i \in k$, then $\langle g^2 \rangle$ realizes $4\mathsf{A}$. If $\epsilon_8 \in k$, then $\langle g \rangle$ realizes $8\mathsf{A}$.
\end{proof}

Over $\bar{k}$, the group $12\mathsf{A}$ acts by automorphisms on the surface given by
\begin{equation}
    \label{eq:H33C4}
        x_0^3 + x_1^3 + x_2^3 + x_3^3 + 3(\sqrt{3}-1)x_0x_1x_2 = 0
\end{equation}
and this surface is unique up to projective equivalence (see Lemma 12.15 of \cite{DolDun18}).

\begin{prop}
\label{prop:12A}
Let $k$ be a field of characteristic zero. Let $G$ be a subgroup of $12\mathsf{A}$ up to conjugacy in $W(\mathsf{E}_6)$.
\begin{itemize}
    \item[(i)] Suppose $\epsilon_{12} \in k$. Then $G$ is realized over $k$.

    \item[(ii)] Suppose $\epsilon_3 \not\in k$ and $i \in k$. Then $G$ is realized over $k$ if and only if $G$ is a subgroup of $4\mathsf{A}$ or $3\mathsf{D}$.

    \item[(iii)] Suppose $\epsilon_3 \in k$ and $i \not\in k$. Then $G$ is realized over $k$ if and only if $G$ is a subgroup of $3\mathsf{A}$.

    \item[(iv)] Suppose $\epsilon_3 \not\in k$ and $i \not\in k$. Then $G$ is realized over $k$ if and only if $G$ is a subgroup of $3\mathsf{D}$.
\end{itemize}
\end{prop}
\begin{proof}
    (i) Let $X$ be the surface given by Equation \ref{eq:H33C4}. Then $\Aut(\overline{X})$ is isomorphic to $\mathcal{H}_3(3) \rtimes C_4$ and generated by the morphisms
    \begin{align*}
        &g_1: (x_0:x_1:x_2:x_3) \mapsto (x_0:x_1:x_2:\epsilon_3 x_3) \\
        &g_2: (x_0:x_1:x_2:x_3) \mapsto (x_1:x_2:x_0:x_3) \\
        &g_3: (x_0:x_1:x_2:x_3) \mapsto (\epsilon_3 x_0: \epsilon_3^2 x_1: x_2: x_3) \\
        &g_4: (x_0:x_1:x_2:x_3) \mapsto \left(x_0 + x_1 + x_2: x_0 + \epsilon_3 x_1 + \epsilon_3^2 x_2 : x_0 + \epsilon_3^2 x_1 + \epsilon_3 x_2 : \sqrt{3}x_3 \right)
    \end{align*}
    defined over $k$. 
    
    Since $12\mathsf{A}$ is isomorphic to $\mathcal{H}_3(3) \rtimes C_4$, consider the exact sequence
    \begin{equation*}
        1 \to \mathcal{H}_3(3) \to 12\mathsf{A} \xrightarrow[]{\rho} C_4 \to 1.
    \end{equation*}
    By the Schur-Zassenhaus theorem, $G$ is always isomorphic to $(G \cap \mathcal{H}_3(3)) \rtimes \rho(G)$. The elements of order 4 in $\mathcal{H}_3(3) \rtimes C_4$ are of type $4A$, so $\rho(G) \cong C_4$ implies $i \in k$ by Lemma \ref{lem:elements}. Moreover, $\mathcal{H}_3(3)$ has 2 elements of type $3A$ and 24 elements of type $3D$. The reader can reference Table 7 of \cite{DolDun18} for a count of the elements in $12\mathsf{A}$ of each conjugacy type.
    
    (ii) Since $\epsilon_3 \not\in k$, Lemma \ref{lem:elements} implies $G$ does not contain an element of type $3A$. Then $G \cap \mathcal{H}_3(3)$ only contains elements of type $3D$, and thus $G \cap \mathcal{H}_3(3)$ is trivial or generated by a single element of type $3D$. If $G \cap \mathcal{H}_3(3)$ is trivial, then $G$ is isomorphic to $\rho(G)$ and therefore a subgroup of $4\mathsf{A}$ up to conjugacy in $W(\mathsf{E}_6)$. Suppose $G \cap \mathcal{H}_3(3)$ is isomorphic to $C_3$. If $\rho(G)$ is $C_4$, then $C_4$ is a 2-Sylow subgroup, so $G$ is cyclic or dicyclic. However, neither case is possible since $12\mathsf{A}$ does not contain a dicyclic group of order 12, and $G$ cannot contain an element of type 12A by Lemma \ref{lem:elements}. Therefore, $\rho(G)$ is trivial or $C_2$. Every element of order 6 in $12\mathsf{A}$ is type $6A$, so if $\rho(G)$ is $C_2$, then $G$ is isomorphic to $S_3$. In this case $G$ is $3\mathsf{D}$ up to conjugacy in $W(\mathsf{E}_6)$.

    (iii) If $i \not\in k$, then $G$ does not contain an element of type $4A$ by Lemma \ref{lem:elements}. So $G$ is a subgroup of $\mathcal{H}_3(3) \rtimes C_2$. Since $\epsilon_3 \in k$, all of $3\mathsf{A}$ is realized on a surface of the form $V(x_0^3 + x_1^3 + x_2^3 + x_3^3 + \lambda x_0x_1x_2)$ for any $\lambda \in k$ by Proposition 10.6 of \cite{DolDun18}. Here $\Aut(\overline{X})$ is generated by $g_1,g_2,g_3$ and the morphism $(x_0:x_1:x_2:x_3) \mapsto (x_1:x_0:x_2:x_3)$.

    (iv) If $\epsilon_3 \not\in k$ and $i \not\in k$ then $G \cap \mathcal{H}_3(3)$ is either $C_3$ or trivial and $\rho(G)$ is either $C_2$ or trivial. Once again this forces $G$ to be a subgroup of $3\mathsf{D}$ up to conjugacy in $W(\mathsf{E}_6)$.
\end{proof}

\subsection{Subgroups of \texorpdfstring{$3\mathsf{C}$}{3C}}

By Lemma 10.14 of \cite{DolDun18}, a smooth cubic surface admitting an action of an element of type $3C$ is projectively equivalent to the Fermat cubic surface over $\bar{k}$. Recall that the \textit{Fermat cubic surface} $X_0$ is given by the equation
\begin{equation}
\label{eq:fermat}
    x_0^3 + x_1^3 + x_2^3 + x_3^3 = 0.
\end{equation} 
The group $3\mathsf{C}$ is realized as $\Aut(\overline{X_0})$, and $3\mathsf{C}$ is isomorphic to $C_3^3 \rtimes S_4$. We identify $C_3^3$ with
\begin{equation*}
    \left\{[a_1,a_2,a_3,a_4] \, \big| \, a_i \in \mathbb{F}_3 \text{ and } \sum a_i = 0 \right\},
\end{equation*}
and we write elements of $\Aut(\overline{X_0})$ as $[a_1,a_2,a_3,a_4] \cdot \sigma$ where $\sigma$ is an element of $S_4$. Then 
\begin{equation*}
    \sigma^{-1}[a_1,a_2,a_3,a_4]\sigma = [a_{\sigma(1)},a_{\sigma(2)},a_{\sigma(3)},a_{\sigma(4)}]
\end{equation*}
Note that our choice of representatives for $C_3^3$ gives a faithful representation of $\Aut(\overline{X_0})$ into $\GL_4(\bar{k})$.

We give two constructions that realize subgroups of $3\mathsf{C}$ as automorphisms of a cubic surface. We will refer to these examples later.

\begin{example}($3\mathsf{C}_1$)
\label{ex:C32D8}
Let $k$ be a field of characteristic zero. Let $G = \langle g_1,g_2,g_3,g_4 \rangle$ be the subgroup of $\PGL_4(k)$ with
\begin{equation*}
    g_1 = \begin{bmatrix} 0 & -1 & 0 & 0 \\ 1 & -1 & 0 & 0 \\ 0 & 0 & 1 & 0 \\ 0 & 0 & 0 & 1 \end{bmatrix},
    \,\,
    g_2 = \begin{bmatrix} 1 & 0 & 0 & 0 \\ 0 & 1 & 0 & 0 \\ 0 & 0 & 0 & -1 \\ 0 & 0 & 1 & -1 \end{bmatrix}, \,\,
    g_3 = \begin{bmatrix} 0 & 0 & 1 & 0 \\ 0 & 0 & 0 & 1 \\ 0 & 1 & 0 & 0 \\ 1 & 0 & 0 & 0 \end{bmatrix}, \,\,
    g_4 = \begin{bmatrix} 0 & 0 & 0 & 1 \\ 0 & 0 & 1 & 0 \\ 0 & 1 & 0 & 0 \\ 1 & 0 & 0 & 0 \end{bmatrix}.
\end{equation*}
Then $G$ is isomorphic to $C_3^2 \rtimes D_8$. Notice that $g_1g_2$ is an element of $G$ of type $3C$. If $G$ acts on a smooth cubic surface it must be projectively equivalent to the Fermat cubic surface over $\bar{k}$, and the resulting image of $G$ in $W(\mathsf{E}_6)$ is contained in $3\mathsf{C}$. Notice that $G$ acts by automorphisms on the surface given by
\begin{equation}
\label{eq:C32D8}
    2(x_0^3 + x_1^3 + x_2^3 + x_3^3) - 3(x_0^2x_1 + x_0x_1^2 + x_2^2x_3 + x_2x_3^2) = 0.
\end{equation}
There is a unique subgroup of $3\mathsf{C}$ isomorphic to $C_3^2 \rtimes D_8$ up to conjugacy in $W(\mathsf{E}_6)$ --- in fact, up to conjugacy in $3\mathsf{C}$. We call this subgroup $3\mathsf{C}_1$. We have shown that $3\mathsf{C}_1$ acts by automorphisms on the smooth cubic surface above over any field of characteristic zero.
\end{example}

\begin{remark}
\label{rem:C32D8}
    There is a Galois cohomological approach to realizing $3\mathsf{C}_1$ when $\epsilon_3 \not\in k$. Let $X_0$ be $V(x_0^3 + x_1^3 + x_2^3 + x_3^3)$ over a field $k$ of characteristic zero with $\epsilon_3 \not\in k$. Recall that the forms of $X_0$ up to isomorphism over $k$ are in bijection with $H^1(\bar{k}/k,\Aut(\overline{X_0}))$ (see section III.1 of \cite{Ser97}). We define a cocycle $c \in Z^1(\bar{k}/k,\Aut(\overline{X_0}))$ by
    \begin{equation*}
        c_{\gamma} = \begin{cases}
            \id, &\text{ if } \gamma(\epsilon_3) = \epsilon_3 \\
            (12)(34), &\text{ if } \gamma(\epsilon_3) = \epsilon_3^2
        \end{cases}
    \end{equation*}
    Twisting by $c$, we obtain a form ${}_c X_0$ of the Fermat cubic surface. The action of $\Gamma := \Gal(\bar{k}/k)$ on ${}_c \Aut(\overline{X_0})$ is defined by ${}^{\gamma'} g = c_{\gamma} \cdot {}^{\gamma} g \cdot c_{\gamma}^{-1}$, and $\Aut({}_c X_0) = ({}_c \Aut(\overline{X_0}))^{\Gamma}$. Computing $({}_c \Aut(\overline{X_0}))^{\Gamma}$, we find
    \begin{equation*}
    \Aut({}_c X_0) = \langle [1,2,0,0], [0,0,1,2],(1324),(12) \rangle,
    \end{equation*}
    a representative for $3\mathsf{C}_1$. 
\end{remark}

\begin{example}($3\mathsf{C}_2$)
\label{ex:C3C4}
Let $k$ be a field of characteristic zero. Consider the subgroup $G = \langle [1,2,1,2], (1234) \rangle$ of $\Aut(\overline{X_0})$. Then $G$ is isomorphic to the dicyclic group of order 12, which we write as $\Dic_{12}$. This is the unique subgroup of $\Aut(\overline{X_0})$ isomorphic to $\Dic_{12}$ up to conjugacy in $\Aut(\overline{X_0})$. We use $3\mathsf{C}_2$ to refer to the corresponding conjugacy class of subgroups in $W(\mathsf{E}_6)$. Now assume $x^2 + y^2 = -3$ has a solution over $k$. Under this assumption, we can construct a faithful representation of $\Dic_{12}$ in $\GL_2(k)$. Consider the presentation $\Dic_{12} = \langle r,s \mid r^3 = \id, \, s^4 = \id, \, srs^{-1} = r^{-1} \rangle$. Let $\alpha,\beta \in k$ with $\alpha^2 + \beta^2 = -3$. We have a faithful representation of $\Dic_{12}$ in $\GL_4(k)$ defined by
\begin{equation*}
    r \mapsto \begin{bmatrix} 0 & -1 & 0 & 0 \\ 1 & -1 & 0 & 0 \\ 0 & 0 & 0 & -1 \\ 0 & 0 & 1 & -1 \end{bmatrix}, \quad
    s \mapsto \begin{bmatrix} 0 & 1 & 0 & 0 \\ 1 & 0 & 0 & 0 \\ 0 & 0 & \dfrac{\alpha-1}{\beta} & \dfrac{-\alpha-1}{\beta} \\ 0 & 0 & \dfrac{-2}{\beta} & \dfrac{-\alpha + 1}{\beta} \end{bmatrix}.
\end{equation*}
This representation is isomorphic to the representation defined by $r \mapsto [1,2,1,2]$ and $s \mapsto (1234)$ with image in $\Aut(\overline{X_0})$. The new representation is defined over $k$ and acts on the smooth cubic surface given by
\begin{align}
\label{eq:C3C4}
    8(x_0^3 + x_1^3) - 12(x_0^2x_1 &+ x_0x_1^2) + (\alpha - 1)(x_0x_2^2 + x_1x_3^2) \\
    &- (\alpha + 1)(x_1x_2^2 + 2x_0x_2x_3) + 4x_1x_2x_3 + 2x_0x_3^2 = 0 \nonumber
\end{align}
also defined over $k$. Note that this surface is projectively equivalent over $\bar{k}$ to the Fermat cubic surface. Therefore, $3\mathsf{C}_2$ acts by automorphisms on a smooth cubic surface when $x^2 + y^2 = -3$ has a solution over $k$.
\end{example}

\begin{prop}
\label{prop:C3C4}
Let $k$ be a field of characteristic zero. Then $3\mathsf{C}_2$ acts by automorphisms on a smooth cubic surface over $k$ if and only if $x^2 + y^2 = -3$ has a solution over $k$.
\end{prop}
\begin{proof}
    Example \ref{ex:C3C4} shows that $3\mathsf{C}_2$ is realized when $x^2 + y^2 = -3$ has a solution over $k$. If $\epsilon_3 \in k$, then $3\mathsf{C}_2$ acts on the Fermat cubic surface, and $(1+2\epsilon_3^2) + 0^2 = -3$.

    So we may assume that $\epsilon_3 \not\in k$. Let $\Gamma := \Gal(\bar{k}/k)$. Suppose $3\mathsf{C}_2$ acts on a smooth cubic surface $X$. Since $3\mathsf{C}_2$ contains an element of type $3C$, we know $X$ is projectively equivalent over $\bar{k}$ to the Fermat cubic surface $X_0$. The $k$-forms of $X_0$ are in bijection with $H^1(\bar{k}/k,\Aut(\overline{X_0}))$. Therefore, $X$ is isomorphic to a twist of $X_0$ by some cocycle $c \in Z^1(\bar{k}/k, \Aut(\overline{X_0}))$. Note that if $G_1$ and $G_2$ are conjugate subgroups of $\Aut(\overline{X_0})$ and $G_1 \subseteq ({}_c \Aut(\overline{X_0}))^{\Gamma}$ for a cocycle $c \in Z^1(\bar{k}/k,\Aut(\overline{X_0}))$, then there is a cohomologous cocycle $c'$ with $G_2 \subseteq ({}_{c'} \Aut(\overline{X_0}))^{\Gamma}$.
    Without loss of generality we choose the representative $\langle [1,2,1,2], (1234) \rangle$ for $3\mathsf{C}_2$ in $\Aut(\overline{X_0})$.

    The $\Gamma$-action on ${}_c \Aut(\overline{X_0})$ is defined ${}^{\gamma'} g = c_{\gamma} \cdot {}^\gamma g \cdot c_{\gamma}^{-1}$. We calculate that $\langle [1,2,1,2], (1234) \rangle$ is contained in $({}_c \Aut(\overline{X_0}))^{\Gamma}$ if and only if
    \begin{equation*}
        c_{\gamma} \in \begin{cases}
            \{\id, (13)(24)\}, &\text{if } \gamma(\epsilon_3) = \epsilon_3 \\
            \{(1234), (4321)\}, &\text{if } \gamma(\epsilon_3) = \epsilon_3^2
        \end{cases}
    \end{equation*}
    Since $\epsilon_3 \not\in k$ and the original $\Gamma$-action on $\langle (1234) \rangle$ is trivial, $c: \Gamma \to \langle (1234) \rangle$ is a surjective group homomorphism. Thus $\Gamma / \ker(c)$ is isomorphic to $C_4$. Now $\ker(c) = \Gal(\bar{k}/F)$ with $\Gal(\bar{k}/F) \subset \Gal(\bar{k}/k(\epsilon_3))$. We have inclusions $k \subset k(\epsilon_3) \subset F \subset \bar{k}$ with $\Gal(k(\epsilon_3)/k) \cong C_2$ and $\Gal(F/k) \cong C_4$. By Theorem 2.2.5 of \cite{JenLedYui02}, the field $k(\epsilon_3)$ can be embedded in a $C_4$-extension of $k$ if and only if $x^2 + y^2 = -3$ has a solution over $k$.
\end{proof}

Consider an arbitrary semidirect product of groups $A \rtimes B$. By Proposition IV.2.3 of \cite{Bro82}, the splittings $B \xhookrightarrow{} A \rtimes B$ up to conjugacy by an element of $A$ are in bijection with the elements of $H^1(B, A)$.

\begin{lem}
\label{lem:splittings}
Let $C_3^3 := \left\{[a_1,a_2,a_3,a_4] \, | \, a_i \in \mathbb{F}_3 \text{ and } \sum a_i = 0\right\}$ be an $\mathbb{F}_3 S_{4}$-submodule of the permutation action on $\mathbb{F}_3^4$. For any subgroup $H$ of $S_{4}$, we have $H^1(H, C_3^3) = 0$.
\end{lem}
\begin{proof}
    Recall that $H^n(H,C_3^3)$ admits a primary decomposition
    \begin{equation*}
        H^n(H,C_3^3) = \bigoplus_{p} H^n(H,C_3^3)_{(p)}
    \end{equation*}
    where $p$ ranges over the primes dividing $\abs{H}$. If $A$ is a $p$-Sylow subgroup of $H$, there is an injection $H^n(H,C_3^3)_{(p)} \xhookrightarrow{} H^n(A,C_3^3)$. If $A$ is a $2$-group, then $H^n(A,C_3^3) = 0$ by Corollary III.10.2 of \cite{Bro82}. If $\sigma$ has order $3$ in $S_4$ and $N = 1 + \sigma + \sigma^2$, then 
    \begin{equation*}
    H^1(\langle \sigma \rangle, C_3^3) = \dfrac{\{a \in C_3^3 \mid Na = 0\}}{(\sigma - 1)C_3^3} = 0.
    \end{equation*}
    We conclude that $H^1(H,C_3^3) = 0$.
\end{proof}

\begin{prop}
\label{prop:3C}
Let $k$ be a field of characteristic zero, and let $G$ be a subgroup of $3\mathsf{C}$ up to conjugacy in $W(\mathsf{E}_6)$.
\begin{itemize}
    \item[(i)] Suppose $\epsilon_3 \in k$. Then $G$ is realized over $k$ on the Fermat cubic surface.

    \item[(ii)] Suppose $x^2 + y^2 = -3$ does not have a solution over $k$. Then $G$ is realized over $k$ if and only if $G$ is a subgroup of $4\mathsf{B}$ or $3\mathsf{C}_1$.
    
    \item[(iii)] Suppose $\epsilon_3 \not\in k$, but $x^2 + y^2 = -3$ has a solution over $k$. Then $G$ is realized over $k$ if and only if $G$ is a subgroup of $4\mathsf{B}$, $3\mathsf{C}_1$, or $3\mathsf{C}_2$.
\end{itemize}
\end{prop}
\begin{proof}
    We let $C_3^3 := \{[a_1,a_2,a_3,a_4] \mid a_i \in \mathbb{F}_3 \text{ and } \sum a_i = 0 \}$ and we identify $3\mathsf{C}$ with $C_3^3 \rtimes S_4$ where $S_4$ acts by permuting coordinates. Let $\rho: C_3^3 \rtimes S_4 \to S_4$ be the projection map. We pick a representative for $G$ in $C_3^3 \rtimes S_4$ and obtain the exact sequence
    \begin{equation*}
        1 \to G \cap C_3^3 \to G \to \rho(G) \to 1
    \end{equation*}
    by restricting $\rho$ to $G$. If $\epsilon_3 \in k$, then $G$ is realized on the the Fermat cubic surface since $\Aut(\overline{X_0}) = \Aut(X_0)$ in this case. 
    
    We assume $\epsilon_3 \not\in k$. Then $G$ does not contain an element of type $3A$ by Lemma \ref{lem:elements}, so $G \cap C_3^3 \neq C_3^3$. Up to conjugacy by $S_4$, the subgroups of $C_3^3$ of order 3 are
    \begin{equation*}
        \langle [1,2,0,0] \rangle, \, \langle [1,1,1,0] \rangle, \, \text{and } \langle [1,2,1,2] \rangle.
    \end{equation*}
    The subgroups isomorphic to $C_3^2$ up to conjugacy are the orthogonal complements of the subgroups of order 3 with respect to the dot product pairing on $C_3^3$. Vectors of the form $[1,2,0,0]$, $[1,1,1,0]$, and $[1,2,1,2]$ up to the action of $S_4$ correspond to elements in $W(\mathsf{E}_6)$ of type $3D$, $3A$, and $3C$ respectively.

    \textbf{Case 1}: Suppose $G \cap C_3^3 \cong C_3^2$. Then up to conjugacy by $S_4$, we may assume $G \cap C_3^3 = \langle[1,2,0,0],[0,0,1,2]\rangle$. Now $\rho(G)$ is contained in the stabilizer of $G \cap C_3^3$ in $S_4$. So $\rho(G)$ is a subgroup of $\langle (12), (34), (13)(24) \rangle \cong D_8$. By the Schur-Zassenhaus theorem, there is a splitting $\rho(G) \xhookrightarrow{} C_3^3 \rtimes \rho(G)$. By Lemma \ref{lem:splittings}, this splitting is unique up to $C_3^3$-conjugacy. So we may assume $G$ is contained in $\langle [1,2,0,0],[0,0,1,2],(12),(34),(13)(24) \rangle$, a representative for $3\mathsf{C}_1$.

    \textbf{Case 2}: Suppose $G \cap C_3^3 \cong C_3$. Since $G$ does not contain an element of type $3A$, we can assume $G \cap C_3^3 = \langle [1,2,0,0] \rangle$ or $G \cap C_3^3 = \langle [1,2,1,2] \rangle$ up to conjugacy by an element of $S_4$. If $G \cap C_3^3 = \langle [1,2,0,0] \rangle$, then $\rho(G)$ is contained in $\langle (12), (34) \rangle$. By Lemma \ref{lem:splittings}, the splitting $\rho(G) \xhookrightarrow{} C_3^3 \rtimes \rho(G)$ is unique up to $C_3^3$-conjugacy, so we may assume $G$ is contained in $\langle [1,2,0,0], (12), (34) \rangle$, which is contained in $3\mathsf{C}_1$. 
    
    Now suppose $G \cap C_3^3 = \langle [1,2,1,2] \rangle$. Then $\rho(G)$ is contained in $\langle (1234), (13) \rangle$. If $\rho(G)$ is not $\langle (1234) \rangle$ or $\langle (1234), (13) \rangle$, then $G$ is conjugate to a subgroup of $3\mathsf{C}_1$. Since $G$ contains an element of type $3C$, any smooth cubic surface over $k$ with an action of $G$ must be a $k$-form of $X_0$. If $\rho(G) = \langle (1234) \rangle$, then we have a representative for $3\mathsf{C}_2$, which is realized if and only if $x^2 + y^2 = -3$ has a solution over $k$ by Proposition \ref{prop:C3C4}. Suppose $\rho(G) = \langle (1234), (13) \rangle$ and $G$ is contained in $\Aut({}_c X_0)$ for a cocycle $c \in Z^1(\bar{k}/k, \Aut(\overline{X_0}))$. Then $\im(c)$ is contained in $\langle (13)(24) \rangle$, the centralizer of $\langle (1234), (13) \rangle$ in $\Aut(\overline{X_0})$. Let $\gamma \in \Gal(\bar{k}/k)$ with $\gamma(\epsilon_3) = \epsilon_3^2$. If either $c_{\gamma} = \id$ or $c_{\gamma} = (13)(24)$, then $c_{\gamma} {}^{\gamma} [1,2,1,2] c_{\gamma}^{-1} = [2,1,2,1]$, a contradiction. So $\langle [1,2,1,2],(1234),(13) \rangle$ is not realized over $k$.

    \textbf{Case 3}: Suppose $G \cap C_3^3$ is trivial. Then $G$ is isomorphic to $\rho(G)$, a subgroup of $S_4$. By Lemma \ref{lem:splittings}, we know $G$ is contained in $\langle (1234), (123) \rangle$ up to $C_3^3$-conjugacy, and $\langle (1234), (123) \rangle$ corresponds to $4\mathsf{B}$ in $W(\mathsf{E}_6)$.
\end{proof}

\subsection{Proof of Theorem \ref{thm:main}} Recall that $\mathcal{P}_{3,k}$ is the collection of conjugacy classes of subgroups of $W(\mathsf{E}_6)$ that act by automorphisms on a smooth cubic surface over $k$. We say a group in $\mathcal{P}_{3,k}$ is maximal if it is maximal with respect to inclusion. We are ready to complete the proof of the main theorem.
\begin{proof}
    Propositions \ref{prop:5A}, \ref{prop:8A}, \ref{prop:12A}, and \ref{prop:3C} describe the possible maximal subgroups in the intersection of $\mathcal{P}_{3,k}$ with the downward closure of $5\mathsf{A}$, $8\mathsf{A}$, $12\mathsf{A}$, or $3\mathsf{C}$ respectively. Collecting the maximal groups from these propositions, it is immediate that the potential maximal groups in $\mathcal{P}_{3,k}$ are those listed in the table. The conditions on $k$ and the surfaces in the final column are also immediate from the propositions and examples.

    The groups $5\mathsf{A}$, $3\mathsf{C}$, $12\mathsf{A}$, and $8\mathsf{A}$ are maximal in $\mathcal{P}_{3,k}$ whenever the corresponding field condition is satisfied. The group $4\mathsf{A}$ is maximal in $\mathcal{P}_{3,k}$ if $i \in k$ but $\epsilon_3,\epsilon_8 \not\in k$. The group $3\mathsf{C}_1$ is maximal when $\epsilon_3 \not\in k$. The group $3\mathsf{C}_2$ is maximal when $x^2 +y^2 = -3$ has a solution over $k$ and $\epsilon_3 \not\in k$.
\end{proof}

\begin{remark}
    Notice that if a subgroup $G$ in $W(\mathsf{E}_6)$ acts on a smooth cubic surface over $\bar{k}$, then there is a field $F$ obtained by at most two quadratic extensions of $k$ and a cubic surface defined over $F$ on which $G$ acts. For quartic del Pezzo surfaces, we only need a single quadratic extension (see Theorem 1.1 of \cite{Smi23}).
\end{remark}

\begin{cor}
\label{cor:elements}
Let $k$ be a field of characteristic zero. Let $g$ be an element of $W(\mathsf{E}_6)$ up to conjugacy. Then $g$ acts by automorphisms on a smooth cubic surface over $k$ if and only if $g$ is one of the classes in the table and $k$ contains the corresponding primitive root of unity.
\begin{center}
    \begin{tabular}{|c|c|c|c|c|c|c|c|c|c|c|c|c|c|c|c|c|}
    \hline
    \textbf{Name} & $1A$ & $2A$ & $2B$ & $3A$ & $3C$ & $3D$ & $4A$ & $4B$ & $5A$ & $6A$ & $6C$ & $6E$ & $6F$ & $8A$ & $9A$ & $12A$ \\
    \hline
    \textbf{Root} & {} & {} & {} & $\epsilon_3$ & {} & {} & $i$ & {} & {} & $\epsilon_3$ & $\epsilon_3$ & {} & {} & $\epsilon_8$ & $\epsilon_3$ & $\epsilon_{12}$ \\
    \hline
    \end{tabular}
\end{center}
If a root of unity is not listed, then the class is realized over any field of characteristic zero.
\end{cor}
\begin{proof}
    This follows immediately from Table 9.5 of \cite{Dol12}, Lemma \ref{lem:elements}, and Theorem \ref{thm:main}.
\end{proof}

\section{Rationality}
\label{sec:rat}

For any field $k$ of characteristic zero, let $\Gamma$ be $\Gal(\bar{k} / k)$. Let $X$ be a smooth cubic surface over $k$. Recall that the action of $\Gamma$ on $\Pic \overline{X}$ induces a map $\rho : \Gamma \to W(\mathsf{E}_6)$. We let $\overline{\Gamma}$ be the image of $\Gamma$ under this map. The $\Gamma$-orbits of the lines on $\overline{X}$ are useful in determining the rationality of $X$. 

Any $\Gamma$-orbit of skew lines on $\overline{X}$ can be blown down to obtain a birational morphism defined over $k$ to a del Pezzo surface of degree $3 + n$, where $n$ is the size of the orbit of lines. Moreover, if $X$ is a del Pezzo surface of degree $d \geq 5$, then the $k$-rationality of $X$ is equivalent to $X(k) \neq \emptyset$ by Theorem 29.4 of \cite{Man86}. P. Swinnerton-Dyer showed that every del Pezzo surface of degree five has a $k$-point \cite{SD70}, from which we can conclude that every del Pezzo surface of degree five is $k$-rational.

A surface $X$ over $k$ is \textit{minimal} if every birational morphism $f: X \to X'$ is necessarily an isomorphism. A del Pezzo surface $X$ is minimal if there are no $\Gamma$-orbits of skew lines on $\overline{X}$. Combining results of B. Segre \cite{Seg51}, Y. Manin \cite{Man66,Man67}, and V. Iskovskikh \cite{Isk72}, one can ascertain that every minimal del Pezzo surface of degree $d \leq 4$ is not $k$-rational; see Theorem 3.3.1 of \cite{ManTsf86}.

Combining Theorem 29.4 and Theorem 30.1 of \cite{Man86}, we also obtain a characterization for the $k$-unirationality of del Pezzo surfaces of degree $d \geq 3$. A del Pezzo surface $X$ of degree $d \geq 3$ is $k$-unirational if and only if $X(k) \neq \emptyset$. For additional results related to rationality, the reader should consult \cite{Man86}, but the results we have listed here are sufficient for our purposes.

\begin{lem}
\label{lem:twolines}
For any field $k$, a smooth cubic surface over $k$ that contains two skew lines defined over $k$ is $k$-rational.
\end{lem}
\begin{proof}
    Let $X$ be a smooth cubic surface over $k$ with two skew lines defined over $k$. Blowing down along both lines, we obtain a birational map defined over $k$ to a quintic del Pezzo surface, and the quintic del Pezzo surface is $k$-rational.
\end{proof}

\begin{prop}
\label{prop:centralizer}
Suppose $X$ is a smooth cubic surface with $\Aut(X) = \Aut(\overline{X})$. If $\overline{\Gamma}$ is abelian and contained in $\Aut(X)$, as subgroups of $W(\mathsf{E}_6)$, then there exists a $k$-rational smooth cubic surface $Y$ with $\Aut(Y) = \Aut(X)$.
\end{prop}
\begin{proof}
    We construct a $k$-form $Y$ that is $k$-rational with $\Aut(Y) = \Aut(X)$. Since the $\Gamma$-action on $\Aut(\overline{X})$ is trivial, $Z^1(\bar{k}/k, \Aut(\overline{X})) = \Hom(\Gamma, \Aut(\overline{X}))$. We define a cocycle $c: \Gamma \to \Aut(\overline{X})$ by $c_{\gamma} = \bar{\gamma}^{-1}$ where $\bar{\gamma}$ is the image of $\gamma$ in $\overline{\Gamma}$. Notice that $c$ is well-defined since $\overline{\Gamma}$ is abelian and contained in $\Aut(X)$. Let $Y$ be the surface obtained by twisting $X$ by $c$. The $\Gamma$-action on the lines of $Y$ is then ${}^{\gamma'} L = c_{\gamma} \cdot {}^{\gamma} L = \bar{\gamma}^{-1}(\bar{\gamma}(L)) = L$. Since the $\Gamma$-action on the lines of $Y$ is trivial, $Y$ is $k$-rational. Since $\overline{\Gamma}$ is in the centralizer of $\Aut(X)$ in $W(\mathsf{E}_6)$, we have $\Aut(Y) = \Aut(X)$.
\end{proof}

Recall that $X_0$ denotes the Fermat cubic surface $V(x_0^3 + x_1^3 + x_2^3 + x_3^3)$. We describe the 27 lines of $\overline{X_0}$ as follows. We group the 27 lines of $\overline{X_0}$ into the following three sets:
\begin{equation*}
    \underbrace{\{x_0 + \epsilon_3^n x_1 = x_2 + \epsilon_3^m x_3 = 0\}}_0, \, \underbrace{\{x_0 + \epsilon_3^n x_2 = x_1 + \epsilon_3^m x_3 = 0\}}_1, \, \underbrace{\{x_0 + \epsilon_3^n x_3 = x_1 + \epsilon_3^m x_2 = 0\}}_2
\end{equation*}
The lines of $\overline{X_0}$ are then given by $\{L_{i,n,m} \mid i,n,m \in \mathbb{F}_3\},$ where $i$ denotes the set to which the line belongs. If $k$ does not contain $\epsilon_3$, then the $\Gamma$-action on the lines of $\overline{X_0}$ is determined by the rule
\begin{equation*}
    {}^{\gamma} L_{i,n,m} = \begin{cases}
        L_{i,n,m} &\text{ if } \gamma(\epsilon_3) = \epsilon_3 \\
        L_{i,2n,2m} &\text{ if } \gamma(\epsilon_3) = \epsilon_3^2
    \end{cases}
\end{equation*}
for all $\gamma$ in $\Gamma$. If $k$ does contain $\epsilon_3$, then the $\Gamma$-action on the lines of $\overline{X_0}$ is trivial. We record conditions for two lines $L_{i,n,m}$ and $L_{i',n',m'}$ to intersect:
\begin{itemize}
    \item[(i)] If $i = i'$, then $n = n'$ or $m = m'$.

    \item[(ii)] If $i = 0$ and $i' = 1$, then $m' - m + n - n' = 0$ in $\mathbb{F}_3$.

    \item[(iii)] If $i = 0$ and $i' = 2$, then $m' + m + n - n' = 0$ in $\mathbb{F}_3$.

    \item[(iv)] If $i = 1$ and $i' = 2$, then $m' - m + n' - n = 0$ in $\mathbb{F}_3$.
\end{itemize}
We can now quickly determine if a $\Gamma$-orbit of lines is pairwise skew.

\begin{prop}
\label{prop:Dic12rationality}
Let $k$ be a field of characteristic zero. If $k$ does not contain $\epsilon_3$, then any smooth cubic surface over $k$ on which $3\mathsf{C}_2$ acts by automorphisms is neither $k$-rational nor stably $k$-rational.
\end{prop}
\begin{proof}
    Let $X$ be a smooth cubic surface on which $3\mathsf{C}_2$ acts by automorphisms. Since $3\mathsf{C}_2$ contains an element of type $3C$, we know $X$ is a $k$-form of $X_0$. Without loss of generality, we pick the representative $\langle [1,2,1,2], (1234) \rangle$ for $3\mathsf{C}_2$ in $\Aut(\overline{X_0})$. By the proof of Proposition \ref{prop:C3C4}, $X$ is a twist of $X_0$ by a cocycle $c \in Z^1(\bar{k}/k, \Aut(\overline{X_0}))$ with
    \begin{equation*}
        c_{\gamma} \in \begin{cases}
            \{\id, (13)(24)\}, &\text{if } \gamma(\epsilon_3) = \epsilon_3 \\
            \{(1234), (4321)\}, &\text{if } \gamma(\epsilon_3) = \epsilon_3^2
        \end{cases}
    \end{equation*}
    and $\im(c) = \langle (1234) \rangle$. The $\Gamma$-action on the lines of $\overline{X}$ is given by ${}^{\gamma'} L = c_{\gamma} \cdot {}^{\gamma} L$. The orbits of this action are
    \begin{align*}
        &\{L_{1,0,0}\}, \, \{L_{0,0,0}, L_{2,0,0}\}, \, \{L_{0,1,1}, L_{2,1,2}\}, \, \{L_{0,2,2}, L_{2,2,1}\}, \, \{L_{1,0,1}, L_{1,1,0}, L_{1,0,2}, L_{1,2,0}\}, \\
        &\{L_{0,1,2}, L_{0,2,1}, L_{2,2,2}, L_{2,1,1}\}, \, \{L_{1,1,1}, L_{1,1,2}, L_{1,2,2}, L_{1,2,1}\}, \\
        &\{L_{0,0,1}, L_{2,1,0}, L_{0,1,0}, L_{2,0,2}\}, \, \{L_{0,0,2}, L_{2,2,0}, L_{0,2,0}, L_{2,0,1}\}
    \end{align*}
    One checks that each orbit contains a pair of intersecting lines except for the unique fixed line $L_{1,0,0}$. Moreover, $L_{1,0,0}$ is fixed by $\langle [1,2,1,2], (1234) \rangle$, so blowing down along $L_{1,0,0}$ yields a $k$-birational map to a quartic del Pezzo surface $Y$ over $k$, and $\Aut(Y)$ contains $\Dic_{12}$. The groups that act by automorphisms on quartic del Pezzo surfaces are studied in \cite{Smi23}. By Propositions 6.2 and 6.5 of \cite{Smi23}, $Y$ must not be $k$-rational or stably $k$-rational. Thus, $X$ is not $k$-rational or stably $k$-rational.
\end{proof}

\begin{remark}
    The proof of Proposition \ref{prop:Dic12rationality} suggests another method to construct a smooth cubic surface with an action of $3\mathsf{C}_2$. Suppose $\epsilon_3 \not\in k$ and suppose there exist $\alpha,\beta \in k$ with $\alpha^2 + \beta^2 = -3$. By Example 5.4 of \cite{Smi23}, the quartic del Pezzo surface $X$ obtained by intersecting the following quadrics
    \begin{align*}
        q_1 &= u_0^2 - 8u_1^2 + u_2^2 + 2u_3^2 + 2(1-\alpha)u_3u_4 - (\alpha + 1)u_4^2 \\
        q_2 &= 2u_0^2 - 4u_1^2 - u_2^2 + (\alpha + 1)u_3^2 + 4u_3u_4 + (1-\alpha)u_4^2
    \end{align*}
    in $\mathbb{P}_k^4$ has automorphisms
    \begin{align*}
        &g_1: (u_0:...:u_4) \mapsto \left(u_2:\dfrac{1}{2}u_0: 2u_1: u_4: -u_3 - u_4 \right) \\
        &g_2: (u_0:...:u_4) \mapsto \left(2u_1: \dfrac{1}{2}u_0: u_2: \dfrac{-\alpha - 1}{\beta}u_3 - \dfrac{2}{\beta}u_4: \dfrac{\alpha - 1}{\beta}u_3 + \dfrac{\alpha + 1}{\beta}u_4 \right).
    \end{align*}
    Notice that $\langle g_1, g_2 \rangle$ is isomorphic to $\Dic_{12}$ and fixes the point $(2:1:2:0:0)$ on $X$. None of the exceptional curves of $\overline{X}$ pass through $(2:1:2:0:0)$, so blowing up $(2:1:2:0:0)$, we obtain a smooth cubic surface $Y$ over $k$ with $\Dic_{12} \xhookrightarrow{} \Aut(Y)$. Since there is a unique conjugacy class of subgroups isomorphic to $\Dic_{12}$ in $W(\mathsf{E}_6)$, the surface $Y$ admits an action of $3\mathsf{C}_2$.
\end{remark}

\subsection*{Proof of Theorem \ref{thm:rationality}} We are ready to prove the second main theorem.
\begin{proof}
    (i) We first show that if $G$ is not conjugate to $3\mathsf{C}_2$, then $G$ acts on a $k$-rational cubic surface. It suffices to show that every group apart from $3\mathsf{C}_2$ in the table of Theorem \ref{thm:main} acts on a $k$-rational surface when the appropriate condition on $k$ is satisfied. 

    \textbf{Case 1}: Suppose $G$ is $5\mathsf{A}$. Then $G$ acts on the Clebsch cubic surface given by Equation \ref{eq:clebsch}. The Clebsch cubic surface contains the lines $l_1 = (a:-a:b:-b)$ and $l_2 = (0:a:b:-a)$. These two lines are skew and defined over $k$, so the Clebsch cubic surface is $k$-rational by Lemma \ref{lem:twolines}.

    \textbf{Case 2}: Suppose $G$ is $3\mathsf{C}$. Then $\epsilon_3 \in k$, and $G$ acts on the Fermat cubic surface $X_0$. Since $\epsilon_3 \in k$, all 27 of the lines on $\overline{X_0}$ are defined over $k$, so $X_0$ is $k$-rational.

    \textbf{Cases 3 and 4}: Suppose $G$ is $12\mathsf{A}$ or $8\mathsf{A}$. If $G$ acts by automorphisms on $X$, then $\Aut(X) = \Aut(\overline{X})$. We compute that the centralizer of $\Aut(X)$ in $W(E_6)$ is abelian and contained in $\Aut(X)$, so $\overline{\Gamma}$ is abelian and contained in $\Aut(X)$ in either case. By Proposition \ref{prop:centralizer}, there exists a $k$-rational smooth cubic surface with an action of $G$.

    \textbf{Case 5}: Suppose $G$ is $4\mathsf{A}$. Then $i \in k$, and by Lemma 11.4 of \cite{DolDun18}, $G$ acts by automorphisms via the diagonal matrix $[i,-1,1,1]$ on any surface of the form
    \begin{equation}
        x_0^2x_1 + x_1^2x_3 + x_2^3 - (1 + \alpha)x_2^2x_3 + \alpha x_2x_3^2 = 0
    \end{equation}
    with $\alpha \in k$. The surface is smooth if and only if $\alpha$ is not equal to $0$ or $1$. Any smooth surface in the family contains the line $l_1: \{x_1 = x_2 - x_3 = 0\}$ defined over $k$ as well as the line
    \begin{equation*}
        l_2: \left(1, t, \dfrac{1-\sqrt{\alpha}}{\alpha-1} t, \sqrt{\dfrac{\sqrt{\alpha} + 1}{\alpha}} + \dfrac{\alpha - \sqrt{\alpha}}{\alpha^2-\alpha}t \right)
    \end{equation*}
    described on the affine chart $x_0 \neq 0$. Setting $\alpha = 9$, the lines $l_1$ and $l_2$ are skew and defined over $k$. By Lemma \ref{lem:twolines}, the surface is $k$-rational.

    \textbf{Case 6}: Suppose $G$ is $3\mathsf{C}_1$. Appealing to Remark \ref{rem:C32D8}, $G$ acts on the twisted surface ${}_c X_0$ where $c \in Z^1(\bar{k}/k,\Aut(\overline{X_0}))$ is defined by 
    \begin{equation*}
        c_{\gamma} = \begin{cases}
            \id, &\text{ if } \gamma(\epsilon_3) = \epsilon_3 \\
            (12)(34), &\text{ if } \gamma(\epsilon_3) = \epsilon_3^2
        \end{cases}
    \end{equation*}
    The twisted $\Gamma$-action on the lines of $\overline{{}_c X_0}$ is defined ${}^{\gamma'} L = c_{\gamma} \cdot {}^\gamma L$. The lines corresponding to $L_{0,0,0}$ and $L_{0,1,2}$ on the twisted surface are skew and fixed by the twisted $\Gamma$-action. We conclude that ${}_c X_0$ is $k$-rational by Lemma \ref{lem:twolines}. 
    
    (ii) If $k$ contains $\epsilon_3$, then $3\mathsf{C}_2$ acts on the $k$-rational Fermat cubic surface. (iii) If $x^2 + y^2 = -3$ does not have a solution over $k$, then $3\mathsf{C}_2$ does not act on any smooth cubic surface over $k$ by Theorem \ref{thm:main}. The fact that the conditions (i), (ii), and (iii) force $G$ not to act on a $k$-rational or stably $k$-rational smooth cubic surface is the content of Proposition \ref{prop:Dic12rationality}.
\end{proof}

\begin{cor}
\label{cor:unirational}
Let $k$ be a field of characteristic zero. Every subgroup of $W(\mathsf{E}_6)$ that acts on a smooth cubic surface over $k$ acts on a $k$-unirational smooth cubic surface over $k$.
\end{cor}
\begin{proof}
Suppose $G \subseteq W(\mathsf{E}_6)$ acts on a smooth cubic surface. If the three conditions of Theorem \ref{thm:rationality} are not met, then $G$ acts on a $k$-rational, and thus $k$-unirational, smooth cubic surface over $k$. If the three conditions are satisfied, let $X$ be a cubic surface on which $3\mathsf{C}_2$ acts. In the proof of Proposition \ref{prop:Dic12rationality}, we showed that the line $L_{1,0,0}$ is fixed by the $\Gamma$-action on the lines of $X$, so $X(k) \neq \emptyset$. It follows that $X$ is $k$-unirational.
\end{proof}

\begin{remark}
\label{rem:minimal}
The analogue of Corollary \ref{cor:unirational} for quartic del Pezzo surfaces was proven in \cite{Smi23}. Also notice that if a subgroup $G$ of $W(\mathsf{E}_6)$ acts on a smooth cubic surface over $k$, then we can always find a non-minimal cubic surface on which $G$ acts. This is not the case for group actions on quartic del Pezzo surfaces (see Proposition 6.2 of \cite{Smi23}).
\end{remark}

\bibliographystyle{alpha}
\bibliography{mybibliography}

\end{document}